\documentclass[11pt]{article}

\usepackage{amsmath}
\usepackage{amsthm}
\usepackage{amssymb}
\usepackage{latexsym}
\usepackage{mathabx}
\usepackage{pstricks}
\usepackage{graphicx, pst-plot, pst-node, pst-text, pst-tree}
\usepackage{titlefoot}
\usepackage{enumerate}
\usepackage{titlesec}
\usepackage{units}
\usepackage[small,it]{caption}
\usepackage{rotating}

\usepackage{color}
\definecolor{lightgray}{rgb}{0.9, 0.9, 0.9}
\definecolor{darkgray}{rgb}{0.7, 0.7, 0.7}
\definecolor{darkblue}{rgb}{0, 0, .4}

\usepackage[bookmarks]{hyperref}
\hypersetup{
	colorlinks=true,
	linkcolor=darkblue,
	anchorcolor=darkblue,
	citecolor=darkblue,
	urlcolor=darkblue,
	pdfpagemode=UseThumbs,
	pdftitle={Maximal Independent Sets and Separating Covers},
	pdfsubject={Combinatorics},
	pdfauthor={Vincent Vatter},
	pdfkeywords={integer complexity, maximal independent set, separating covers}
}

\newtheorem{theorem}{Theorem}
\newtheorem{proposition}[theorem]{Proposition}

\newcounter{todocounter}

\newcommand{\minisec}[1]{\bigskip\noindent{\bf #1.}}

\makeatletter
\newfont{\footsc}{cmcsc10 at 8truept}
\newfont{\footbf}{cmbx10 at 8truept}
\newfont{\footrm}{cmr10 at 10truept}
\renewenvironment{abstract}%
		{
		  \begin{list}{}%
		     {\setlength{\rightmargin}{1in}%
		      \setlength{\leftmargin}{1in}}%
		   \item[]\ignorespaces\begin{small}}%
		 {\end{small}\unskip\end{list}}

\datefoot{\today}
\amssubj{05C35, 05C69, 05D05, 11A99}

\newpagestyle{main}[\small]{
	\headrule
	\sethead[\usepage][][]
	{\sc Maximal Independent Sets and Separating Covers}{}{\usepage}}

\title{\sc{Maximal Independent Sets and Separating Covers}}
\author{\sc{Vincent Vatter}\\
\small Department of Mathematics\\[-1pt]
\small University of Florida\\[-1pt]
\small Gainesville, Florida USA\\[-10pt]}

\titleformat{\section}
	{\large\sc}
	{\thesection.}{1em}{}	

\titleformat{\subsection}
	{\normalsize\sc}
	{\thesubsection.}{1em}{}	

\date{}

\begin{document}
\maketitle

\pagestyle{main}

\begin{abstract}
In 1973, Katona raised the problem of determining the maximum number of subsets in a separating cover on $n$ elements.  The answer to Katona's question turns out to be the inverse to the answer to a much simpler question: what is the largest integer which is the product of positive integers with sum $n$?  We give a combinatorial explanation for this relationship, via Moon and Moser's answer to a question of Erd\H{o}s: how many maximal independent sets can a graph on $n$ vertices have?  We conclude by showing how Moon and Moser's solution also sheds light on a problem of Mahler and Popken's about the complexity of integers.
\end{abstract}

\maketitle

\renewcommand{\S}{\mathcal{S}}
\newcommand{\M}{\mathcal{M}}
\newcommand{\C}{\mathcal{C}}
\newcommand{\Cong}{\equiv}
\newcommand{\Mod}{\mathop{\rm mod}\nolimits}
\newcommand{\fproduct}{\ell}
\newcommand{\fkatona}{s}
\newcommand{\fmoon}{g}

\section{Introduction}

We begin with a simply stated problem, which has made numerous appearances in mathematics competitions:%
\footnote{In particular, the 1976 IMO asked for the $n=1976$ case, the 1979 Putnam asked for the $n=1979$ case, and on April 23rd 2002, the 3rd Community College of Philadelphia Colonial Mathematics Challenge asked for the $n=2002$ case.} %
what is the largest number which can be written as the product of positive integers that sum to $n$?

We denote this number by $\fproduct(n)$.  A moment's thought shows that one should use as many $3$s as possible; if $m\ge 5$ appears in the product then it can be replaced by $3(m-3)>m$, and while $2$s and $4$s can occur in the product, the latter can occur at most once since $4\cdot 4<2 \cdot 3 \cdot 3$ and the former at most twice since $2 \cdot 2 \cdot 2<3 \cdot 3$.  This shows that for $n\ge 2$,
$$
\fproduct(n)
=
\left\{
\begin{array}{ll}
3^i
&
\mbox{if $n=3i$,}\\
4 \cdot 3^{i-1}
&
\mbox{if $n=3i+1$,}\\
2 \cdot 3^i
&
\mbox{if $n=3i+2$,}\\
\end{array}
\right.
$$
while $\fproduct(1)=1$.  Note that it follows from the combinatorial definition of $\fproduct$ that this function is strictly increasing and {\it super-multiplicative\/}, meaning that it satisfies $\fproduct(n_1)\fproduct(n_2)\le\fproduct(n_1+n_2)$.

In 1973, G. O. H. Katona~\cite[Problem 8, p. 306]{katona:combinatorial-s:} posed a problem which looks completely unlike the determination of $\fproduct(n)$.  A {\it separating cover\/}%
\footnote{We make this slight deviation from Katona's original formulation so that $\fkatona(1)=1$.}
over the ground set $X$ is a collection $\S$ of subsets of $X$ which satisfies two properties:
\begin{itemize}
\item the union of the sets in $\S$ is all of $X$, and
\item for every pair of distinct elements $x,y\in X$ there are disjoint sets $S,T\in\S$ with $x\in S$ and $y\in T$.
\end{itemize}
Katona asked about the function
$$
\fkatona(m)=\min\{n : \mbox{there is a separating cover on $m$ elements with $n$ sets\}}.
$$
M.-C. Cai and A. C. C. Yao gave independent solutions several years later.

\begin{theorem}[Cai~\cite{cai:solutions-to-ed:} and Yao~\cite{yao:on-a-problem-of:}, independently]\label{sss}
For all $m\ge 2$,
$$
\fkatona(m)
=
\left\{
\begin{array}{ll}
3i
&
\mbox{if $2 \cdot 3^{i-1}<m\le 3^i$,}\\
3i+1
&
\mbox{if $3^i<m\le 4 \cdot 3^{i-1}$,}\\
3i+2
&
\mbox{if $4 \cdot 3^{i-1}<m\le 2 \cdot 3^i$,}\\
\end{array}
\right.
$$
while $\fkatona(1)=1$.
\end{theorem}

Thus $\fkatona(\fproduct(n))=n$ for all positive integers $n$ --- in other words, $\fkatona$ is a left inverse of $\fproduct$.  Ironically, the question we began with appears at the beginning of R. Honsberger's {\it Mathematical Gems III\/}~\cite{honsberger:mathematical-ge:}, while Katona's problem occurs at the end, where Honsberger describes the proof as ``long and much more complicated than the arguments in the earlier chapters.''  We present a short combinatorial explanation for the equivalence of these two problems.

\section{A Combinatorial Interpretation of ${\fproduct}$}
In order to give a combinatorial explanation for why $\fkatona(\fproduct(n))=n$, we first need a combinatorial interpretation of $\fproduct$.  We use a graph-theoretic interpretation, although several others are available.%
\footnote{Another --- in terms of integer complexity --- is given later in this note.  Additionally, $\fproduct(n)$ is the order of the largest abelian subgroup of the symmetric group of order $n$; see Bercov and Moser~\cite{bercov:on-abelian-perm:}.} %
Let $G$ be a  graph over the vertex set $V(G)$.  A subset $I\subseteq V(G)$ is {\it independent\/} if there is no edge between any two vertices of $I$, and it is a {\it maximal independent set (MIS)\/} if it is not properly contained in any other independent set.  In the 1960s, P. Erd\H{o}s asked how many MISes a graph on $n$ vertices could have, which we define as
$$
\fmoon(n)=\max\{m : \mbox{there is graph on $n$ vertices with $m$ MISes\}}.
$$
Let us denote by $m(G)$ the number of MISes in the graph $G$.  This quantity is particularly easy to compute when $G$ is a disjoint union:

\begin{proposition}\label{mis-product}
The disjoint union of the graphs $G$ and $H$ has $m(G)m(H)$ MISes.
\end{proposition}
\begin{proof}
For any MIS $M$ of this union, $M\cap V(G)$ must be an MIS of $G$ and $M\cap V(H)$ must be an MIS of $H$.  Conversely, if $M_G$ and $M_H$ are MISes of $G$ and $H$, respectively, then $M_G\cup M_H$ is an MIS of the disjoint union of $G$ and $H$.
\end{proof}

Because the complete graph on $n$ vertices has $n$ MISes, Proposition~\ref{mis-product} implies that $\fmoon(n)\ge\fproduct(n)$ for all positive integers $n$; we need only take a disjoint union of edges, triangles, and complete graphs on $4$ vertices to achieve this lower bound.  In 1965, J. W. Moon and L. Moser proved that this is best possible.

\begin{theorem}[Moon and Moser~\cite{moon:on-cliques-in-g:}]\label{mis}
For all positive integers $n$, $\fmoon(n)=\fproduct(n)$.
\end{theorem}

Indeed, Moon and Moser showed that the only extremal graphs (the graphs with $\fmoon(n)$ MISes) are those built by taking disjoint copies of edges, triangles, and complete graphs on $4$ vertices in the quantities suggested by the formula for $\fproduct$.  (In the case $n=3i+1\ge 4$ there are two extremal graphs, one with $i-1$ triangles and two disjoint edges, the other with $i-1$ triangles and a complete graph on $4$ vertices.)

\section{A Short Proof of Theorem~\ref{mis}}
Before demonstrating the relationship between MISes and separating covers, we pause to present a short proof of Moon and Moser's theorem.  First we need a definition: for a set $X\subseteq V(G)$, we denote by $G-X$ the graph obtained by removing the vertices $X$ from $G$ and all edges incident to vertices in $X$.  When $X=\{v\}$, we abbreviate this notation to $G-v$.  Our proof makes extensive use of the following upper bound.

\begin{proposition}\label{mis-recurse}
For any graph $G$ and vertex $v\in V(G)$, we have
$$
m(G)\le m(G-v)+m(G-N[v]),
$$
where $N[v]$ denotes the {\it closed neighborhood\/} of $v$, i.e., $v$ together with its neighbors.
\end{proposition}
\begin{proof}
The map $M\mapsto M-v$ gives a bijection between MISes of $G$ containing $v$ and MISes of $G-N[v]$.  The proof is completed by noting that every MIS of $G$ that does not contain $v$ is also an MIS of $G-v$.
\end{proof}

\newenvironment{mis-proof}{\medskip\noindent {\it Proof of Theorem~\ref{mis}.\/}}{\qed\bigskip}
\begin{mis-proof}
Our proof is by induction on $n$, and we prove the stronger statement which characterizes the extremal graphs.  It is easy to check the theorem for graphs with five or fewer vertices, so take $G$ to be a graph on $n\ge 6$ vertices, and assume the theorem holds for graphs with fewer than $n$ vertices.

If $G$ contains a vertex of degree $0$, that is, an isolated vertex, then clearly $m(G)\le \fmoon(n-1)=\fproduct(n-1)<\fproduct(n)$.  If $G$ contains a vertex $v$ of degree $1$ then, letting $w$ denote the sole vertex adjacent to $v$, we have by Proposition~\ref{mis-recurse} that
$$
m(G)
\le
m(G-w)+m(G-N[w])
\le
2\fproduct(n-2)
=
\left\{\begin{array}{ll}
8\cdot 3^{i-2}&\mbox{if $n=3i$,}\\
4\cdot 3^{i-1}&\mbox{if $n=3i+1$,}\\
2\cdot 3^{i}&\mbox{if $n=3i+2$.}
\end{array}\right.
$$
In all three cases we have an upper bound of at most $\fproduct(n)$, with equality if and only if $n=3i+1$ and $G$ is a disjoint union of $i-1$ triangles and two edges, or $n=3i+2$ and $G$ is a disjoint union of $i$ triangles and an edge.

If $G$ contains a vertex $v$ of degree $3$ or greater, then we have
$$
m(G)
\le
m(G-v)+m(G-N[v])
\le
\fproduct(n-1)+\fproduct(n-4)
=
\left\{\begin{array}{ll}
8\cdot 3^{i-2}&\mbox{if $n=3i$,}\\
4\cdot 3^{i-1}&\mbox{if $n=3i+1$,}\\
16\cdot 3^{i-2}&\mbox{if $n=3i+2$.}
\end{array}\right.
$$
Again, all three cases give an upper bound of at most $\fproduct(n)$, with equality if and only if $n=3i+1$ and $G$ is a disjoint union of $i-1$ triangles together with a complete graph on $4$ vertices.

This leaves us to consider the case where every vertex of $G$ has degree $2$, which implies that $G$ consists of a disjoint union of cycles.  If each of these cycles is a triangle, then $n=3i$ and $G$ is a disjoint union of $i$ triangles, as desired.  Thus we may assume that at least one connected component of $G$ is a cycle of length $j\ge 4$, which we denote by $C_j$.  Our goal in this case is to show that $G$ is not extremal (i.e., $m(G)<\fproduct(n)$), and by the super-multiplicativity of $\fproduct$, it suffices to show that this single cycle of length $j$ is not extremal.  It is easy to check that $m(C_4)=2<4=\fproduct(4)$ and $m(C_5)=5<6=\fproduct(5)$, it therefore suffices to show that $m(C_j)<\fproduct(j)$ for $j\ge 6$.  (In fact, F\"uredi~\cite{furedi:the-number-of-m:} found $m(C_j)$ exactly --- it is the $j$th Perrin number.)  Label the vertices of our cycle on $j\ge 6$ vertices as $u,v,w,\dots$ so that $u$ is adjacent to $v$ which is in turn adjacent to $w$.  By applying Proposition~\ref{mis-recurse} twice, we see that for $j\ge 6$,
\begin{eqnarray*}
m(C_j)
&\le&
m(C_j-w)+m(C_j-N[w])\\
&\le&
m(C_j-w-u)+m(C_j-w-N[u])+m(C_j-N[w])\\
&\le&
2\fproduct(j-3)+\fproduct(j-4),
\end{eqnarray*}
which is strictly less that $3\fproduct(j-3)=\fproduct(j)$, completing the proof.
\end{mis-proof}

\section{A Combinatorial Explanation for ${\fkatona(\fproduct(n))=n}$}
With Moon and Moser's Theorem~\ref{mis} proved, we are now ready to explain the connection to separating covers.  Propositions~\ref{mis2sss} and \ref{sss2mis} illuminate the connection between separating covers and MISes, and then Proposition~\ref{main} gives a combinatorial explanation for why $\fkatona$ is a left inverse of $\fproduct=\fmoon$.

\begin{proposition}\label{mis2sss}
From a graph on $n$ vertices with $m$ MISes one can construct a separating cover on $m$ elements with at most $n$ sets.
\end{proposition}
\begin{proof}
Take $G$ to be a graph with $n$ vertices and $m$ MISes and let $\M$ denote the collection of MISes in $G$.  The separating cover promised consists of the family of sets $\{S_v : v\in V(G)\}$ where
$$
S_v=\{M\in\M : v\in M\}.
$$
Clearly this is a family with $m$ elements (the MISes $\M$) and $n$ (not necessarily distinct) sets (one for each vertex of $G$), and this family covers the set $\M$ because each MIS lies in at least one $S_v$, so it remains to check only that it is separating.  Take distinct sets $M, N\in\M$.  Because $M$ and $N$ are both maximal there is some vertex $u\in M\setminus N$ .  By the maximality of $N$, it must contain a vertex $v$ adjacent to $u$.  Therefore $M\in S_u$, $N\in S_v$, and because $u$ and $v$ are adjacent, $S_u\cap S_v=\emptyset$, completing the proof.
\end{proof}

\begin{proposition}\label{sss2mis}
From a separating cover on $m$ elements with $n$ sets one can construct a graph on $n$ vertices with at least $m$ MISes.
\end{proposition}
\begin{proof}
Let $\S$ be such a cover over the ground set $X$.  We define a graph $G$ on the vertices $\S$ where $S\in\S$ is adjacent to $T\in\S$ if and only if they are disjoint.  For each $x\in X$, the set
$$
I_x=\{S\in\S : x\in S\}
$$
is an independent set in $G$.  For each $x\in X$, choose an MIS $M_x\supseteq I_x$.  We have only to show that these MISes are distinct.  Take distinct elements $x,y\in X$.  Because $\S$ is separating, there are disjoint sets $S,T\in\S$ with $x\in S$ and $y\in T$.  Therefore $S\in M_x$, $T\in M_y$, and since $S$ and $T$ are disjoint they are adjacent in $G$, so $T\notin M_x$, and thus $M_x\neq M_y$.
\end{proof}

\begin{proposition}\label{main}
For all positive integers $m$ and $n$,
\begin{eqnarray*}
\fkatona(m)&=&\min\{n : \fmoon(n)\ge m\},\\
\fmoon(n)&=&\max\{m : \fkatona(m)\le n\}.
\end{eqnarray*}
\end{proposition}
\begin{proof}
First observe that $\fkatona$ and $\fmoon$ are both nondecreasing.  The proof then follows from the two claims
\begin{enumerate}
\item[(1)] If $\fkatona(m)\le n$ then $\fmoon(n)\ge m$, and
\item[(2)] If $\fmoon(n)\ge m$ then $\fkatona(m)\le n$
\end{enumerate}

To prove (1), suppose that $\fkatona(m)\le n$.  Then there is a separating cover with $m$ elements and at most $n$ sets, so by Proposition~\ref{sss2mis}, there is a graph with at most $n$ vertices and at least $m$ MISes.  This and the fact that $g$ is nondecreasing establish that $\fmoon(n)\ge m$.

Now suppose that $\fmoon(n)\ge m$.  Then there is a graph with $n$ vertices and at least $m$ MISes, so by Proposition~\ref{mis2sss}, there is a separating cover with at least $m$ elements and at most $n$ sets.  Because $\fkatona$ is nondecreasing, we conclude that $\fkatona(m)\le n$, proving (2).  
\end{proof}

\section{Integer Complexity}
We conclude with another appearance of $g$.  The {\it complexity\/}, $c(m)$, of the integer $m$ is the least number of $1$s needed to represent it using only $+$s, $\cdot$s, and parentheses.  For example, the complexity of $10$ is $7$, and there are essentially three different minimal expressions:
$$
10=(1+1+1)(1+1+1)+1=(1+1)(1+1+1+1+1)=(1+1)((1+1)(1+1)+1),
$$
Figure~\ref{fig-complexity} shows a plot of the complexities of the first $1000$ integers.

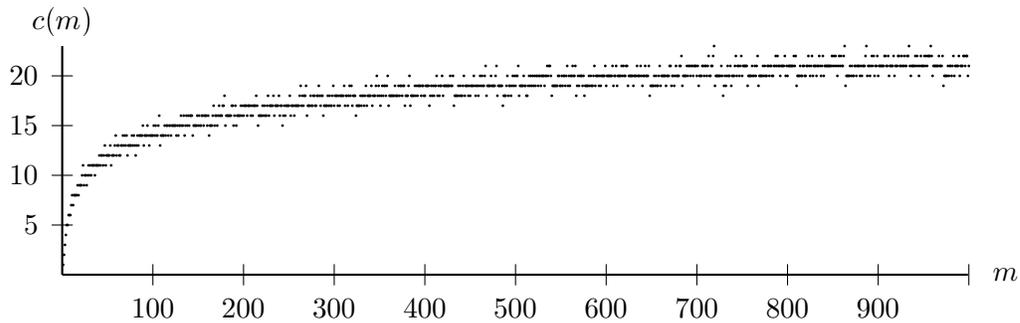
\begin{figure}
\begin{center} 
\savedata{\plotdata}[{{1,1},{2,2},{3,3},{4,4},{5,5},{6,5},{7,6},{8,6},{9,6},{10,7},{11,8},{12,7},{13,8},{14,8},{15,8},{16,8},{17,9},{18,8},{19,9},{20,9},{21,9},{22,10},{23,11},{24,9},{25,10},{26,10},{27,9},{28,10},{29,11},{30,10},{31,11},{32,10},{33,11},{34,11},{35,11},{36,10},{37,11},{38,11},{39,11},{40,11},{41,12},{42,11},{43,12},{44,12},{45,11},{46,12},{47,13},{48,11},{49,12},{50,12},{51,12},{52,12},{53,13},{54,11},{55,12},{56,12},{57,12},{58,13},{59,14},{60,12},{61,13},{62,13},{63,12},{64,12},{65,13},{66,13},{67,14},{68,13},{69,14},{70,13},{71,14},{72,12},{73,13},{74,13},{75,13},{76,13},{77,14},{78,13},{79,14},{80,13},{81,12},{82,13},{83,14},{84,13},{85,14},{86,14},{87,14},{88,14},{89,15},{90,13},{91,14},{92,14},{93,14},{94,15},{95,14},{96,13},{97,14},{98,14},{99,14},{100,14},{101,15},{102,14},{103,15},{104,14},{105,14},{106,15},{107,16},{108,13},{109,14},{110,14},{111,14},{112,14},{113,15},{114,14},{115,15},{116,15},{117,14},{118,15},{119,15},{120,14},{121,15},{122,15},{123,15},{124,15},{125,15},{126,14},{127,15},{128,14},{129,15},{130,15},{131,16},{132,15},{133,15},{134,16},{135,14},{136,15},{137,16},{138,15},{139,16},{140,15},{141,16},{142,16},{143,16},{144,14},{145,15},{146,15},{147,15},{148,15},{149,16},{150,15},{151,16},{152,15},{153,15},{154,16},{155,16},{156,15},{157,16},{158,16},{159,16},{160,15},{161,16},{162,14},{163,15},{164,15},{165,15},{166,16},{167,17},{168,15},{169,16},{170,16},{171,15},{172,16},{173,17},{174,16},{175,16},{176,16},{177,17},{178,17},{179,18},{180,15},{181,16},{182,16},{183,16},{184,16},{185,16},{186,16},{187,17},{188,17},{189,15},{190,16},{191,17},{192,15},{193,16},{194,16},{195,16},{196,16},{197,17},{198,16},{199,17},{200,16},{201,17},{202,17},{203,17},{204,16},{205,17},{206,17},{207,17},{208,16},{209,17},{210,16},{211,17},{212,17},{213,17},{214,18},{215,17},{216,15},{217,16},{218,16},{219,16},{220,16},{221,17},{222,16},{223,17},{224,16},{225,16},{226,17},{227,18},{228,16},{229,17},{230,17},{231,17},{232,17},{233,18},{234,16},{235,17},{236,17},{237,17},{238,17},{239,18},{240,16},{241,17},{242,17},{243,15},{244,16},{245,17},{246,16},{247,17},{248,17},{249,17},{250,17},{251,18},{252,16},{253,17},{254,17},{255,17},{256,16},{257,17},{258,17},{259,17},{260,17},{261,17},{262,18},{263,19},{264,17},{265,18},{266,17},{267,18},{268,18},{269,19},{270,16},{271,17},{272,17},{273,17},{274,18},{275,17},{276,17},{277,18},{278,18},{279,17},{280,17},{281,18},{282,18},{283,19},{284,18},{285,17},{286,18},{287,18},{288,16},{289,17},{290,17},{291,17},{292,17},{293,18},{294,17},{295,18},{296,17},{297,17},{298,18},{299,19},{300,17},{301,18},{302,18},{303,18},{304,17},{305,18},{306,17},{307,18},{308,18},{309,18},{310,18},{311,19},{312,17},{313,18},{314,18},{315,17},{316,18},{317,19},{318,18},{319,19},{320,17},{321,18},{322,18},{323,18},{324,16},{325,17},{326,17},{327,17},{328,17},{329,18},{330,17},{331,18},{332,18},{333,17},{334,18},{335,19},{336,17},{337,18},{338,18},{339,18},{340,18},{341,19},{342,17},{343,18},{344,18},{345,18},{346,19},{347,20},{348,18},{349,19},{350,18},{351,17},{352,18},{353,19},{354,18},{355,19},{356,19},{357,18},{358,19},{359,20},{360,17},{361,18},{362,18},{363,18},{364,18},{365,18},{366,18},{367,19},{368,18},{369,18},{370,18},{371,19},{372,18},{373,19},{374,19},{375,18},{376,19},{377,19},{378,17},{379,18},{380,18},{381,18},{382,19},{383,20},{384,17},{385,18},{386,18},{387,18},{388,18},{389,19},{390,18},{391,19},{392,18},{393,19},{394,19},{395,19},{396,18},{397,19},{398,19},{399,18},{400,18},{401,19},{402,19},{403,19},{404,19},{405,17},{406,18},{407,19},{408,18},{409,19},{410,18},{411,19},{412,19},{413,20},{414,18},{415,19},{416,18},{417,19},{418,19},{419,20},{420,18},{421,19},{422,19},{423,19},{424,19},{425,19},{426,19},{427,19},{428,20},{429,19},{430,19},{431,20},{432,17},{433,18},{434,18},{435,18},{436,18},{437,19},{438,18},{439,19},{440,18},{441,18},{442,19},{443,20},{444,18},{445,19},{446,19},{447,19},{448,18},{449,19},{450,18},{451,19},{452,19},{453,19},{454,20},{455,19},{456,18},{457,19},{458,19},{459,18},{460,19},{461,20},{462,19},{463,20},{464,19},{465,19},{466,20},{467,21},{468,18},{469,19},{470,19},{471,19},{472,19},{473,20},{474,19},{475,19},{476,19},{477,19},{478,20},{479,21},{480,18},{481,19},{482,19},{483,19},{484,19},{485,19},{486,17},{487,18},{488,18},{489,18},{490,19},{491,20},{492,18},{493,19},{494,19},{495,18},{496,19},{497,20},{498,19},{499,20},{500,19},{501,20},{502,20},{503,21},{504,18},{505,19},{506,19},{507,19},{508,19},{509,20},{510,19},{511,19},{512,18},{513,18},{514,19},{515,20},{516,19},{517,20},{518,19},{519,20},{520,19},{521,20},{522,19},{523,20},{524,20},{525,19},{526,20},{527,20},{528,19},{529,20},{530,20},{531,20},{532,19},{533,20},{534,20},{535,21},{536,20},{537,21},{538,21},{539,20},{540,18},{541,19},{542,19},{543,19},{544,19},{545,19},{546,19},{547,20},{548,20},{549,19},{550,19},{551,20},{552,19},{553,20},{554,20},{555,19},{556,20},{557,21},{558,19},{559,20},{560,19},{561,20},{562,20},{563,21},{564,20},{565,20},{566,21},{567,18},{568,19},{569,20},{570,19},{571,20},{572,20},{573,20},{574,19},{575,20},{576,18},{577,19},{578,19},{579,19},{580,19},{581,20},{582,19},{583,20},{584,19},{585,19},{586,20},{587,21},{588,19},{589,20},{590,20},{591,20},{592,19},{593,20},{594,19},{595,20},{596,20},{597,20},{598,20},{599,21},{600,19},{601,20},{602,20},{603,20},{604,20},{605,20},{606,20},{607,21},{608,19},{609,20},{610,20},{611,21},{612,19},{613,20},{614,20},{615,20},{616,20},{617,21},{618,20},{619,21},{620,20},{621,20},{622,21},{623,21},{624,19},{625,20},{626,20},{627,20},{628,20},{629,20},{630,19},{631,20},{632,20},{633,20},{634,21},{635,20},{636,20},{637,20},{638,21},{639,20},{640,19},{641,20},{642,20},{643,21},{644,20},{645,20},{646,20},{647,21},{648,18},{649,19},{650,19},{651,19},{652,19},{653,20},{654,19},{655,20},{656,19},{657,19},{658,20},{659,21},{660,19},{661,20},{662,20},{663,20},{664,20},{665,20},{666,19},{667,20},{668,20},{669,20},{670,21},{671,21},{672,19},{673,20},{674,20},{675,19},{676,20},{677,21},{678,20},{679,20},{680,20},{681,21},{682,21},{683,22},{684,19},{685,20},{686,20},{687,20},{688,20},{689,21},{690,20},{691,21},{692,21},{693,20},{694,21},{695,21},{696,20},{697,21},{698,21},{699,21},{700,20},{701,21},{702,19},{703,20},{704,20},{705,20},{706,21},{707,21},{708,20},{709,21},{710,21},{711,20},{712,21},{713,22},{714,20},{715,20},{716,21},{717,21},{718,22},{719,23},{720,19},{721,20},{722,20},{723,20},{724,20},{725,20},{726,20},{727,21},{728,20},{729,18},{730,19},{731,20},{732,19},{733,20},{734,21},{735,20},{736,20},{737,21},{738,19},{739,20},{740,20},{741,20},{742,21},{743,22},{744,20},{745,21},{746,21},{747,20},{748,21},{749,22},{750,20},{751,21},{752,21},{753,21},{754,21},{755,21},{756,19},{757,20},{758,20},{759,20},{760,20},{761,21},{762,20},{763,20},{764,21},{765,20},{766,21},{767,22},{768,19},{769,20},{770,20},{771,20},{772,20},{773,21},{774,20},{775,21},{776,20},{777,20},{778,21},{779,21},{780,20},{781,21},{782,21},{783,20},{784,20},{785,21},{786,21},{787,22},{788,21},{789,22},{790,21},{791,21},{792,20},{793,21},{794,21},{795,21},{796,21},{797,22},{798,20},{799,21},{800,20},{801,21},{802,21},{803,21},{804,21},{805,21},{806,21},{807,22},{808,21},{809,22},{810,19},{811,20},{812,20},{813,20},{814,21},{815,20},{816,20},{817,21},{818,21},{819,20},{820,20},{821,21},{822,21},{823,22},{824,21},{825,20},{826,21},{827,22},{828,20},{829,21},{830,21},{831,21},{832,20},{833,21},{834,21},{835,22},{836,21},{837,20},{838,21},{839,22},{840,20},{841,21},{842,21},{843,21},{844,21},{845,21},{846,21},{847,21},{848,21},{849,22},{850,21},{851,22},{852,21},{853,22},{854,21},{855,20},{856,21},{857,22},{858,21},{859,22},{860,21},{861,21},{862,22},{863,23},{864,19},{865,20},{866,20},{867,20},{868,20},{869,21},{870,20},{871,21},{872,20},{873,20},{874,21},{875,21},{876,20},{877,21},{878,21},{879,21},{880,20},{881,21},{882,20},{883,21},{884,21},{885,21},{886,22},{887,23},{888,20},{889,21},{890,21},{891,20},{892,21},{893,22},{894,21},{895,22},{896,20},{897,21},{898,21},{899,22},{900,20},{901,21},{902,21},{903,21},{904,21},{905,21},{906,21},{907,22},{908,22},{909,21},{910,21},{911,22},{912,20},{913,21},{914,21},{915,21},{916,21},{917,22},{918,20},{919,21},{920,21},{921,21},{922,22},{923,22},{924,21},{925,21},{926,22},{927,21},{928,21},{929,22},{930,21},{931,21},{932,22},{933,22},{934,23},{935,21},{936,20},{937,21},{938,21},{939,21},{940,21},{941,22},{942,21},{943,22},{944,21},{945,20},{946,21},{947,22},{948,21},{949,21},{950,21},{951,22},{952,21},{953,22},{954,21},{955,22},{956,22},{957,22},{958,23},{959,22},{960,20},{961,21},{962,21},{963,21},{964,21},{965,21},{966,21},{967,22},{968,21},{969,21},{970,21},{971,22},{972,19},{973,20},{974,20},{975,20},{976,20},{977,21},{978,20},{979,21},{980,21},{981,20},{982,21},{983,22},{984,20},{985,21},{986,21},{987,21},{988,21},{989,22},{990,20},{991,21},{992,21},{993,21},{994,22},{995,22},{996,21},{997,22},{998,22},{999,20},{1000,21}}]
\psset{xunit=0.00474525in, yunit=0.0520833in}
\begin{pspicture}(0,-3)(1053.684211,28.80000000)
\psaxes[dy=5,Dy=5,dx=100,Dx=100,showorigin=false](0,0)(999,23)
\rput[c](0,25.4){$c(m)$}
\rput[l](1026.34,0){$m$}
\dataplot[plotstyle=dots,dotstyle=*,dotsize=3\psxunit]{\plotdata}
\end{pspicture}
\end{center}
\caption{The complexities of the first 1000 integers.}\label{fig-complexity}
\end{figure}

This definition was first considered by Mahler and Popken~\cite{mahler:on-a-maximum-pr:}, and while a straightforward recurrence,
$$
c(m)=\min \{c(d)+c(m/d) : d\divides m\}\cup\{c(i)+c(m-i) : 1\le i\le m-1\},
$$
is easy to verify, several outstanding conjectures and questions remain, for which we refer to R. K. Guy~\cite{guy:unsolved-proble:}.  In that article, Guy mentions that J. Selfridge gave an inductive proof of the following result. 

\begin{proposition}[Selfridge~{[unpublished]}]\label{prop-integer-complexity}
The greatest integer of complexity $n$ is $\fmoon(n)$.
\end{proposition}

One direction of Selfridge's proposition is clear: the problem we began with shows that $\fproduct(n)=\fmoon(n)$ has complexity at most $n$.  In a final demonstration of the surprising versatility of Moon and Moser's Theorem~\ref{mis}, we show how it implies the other direction, via the following construction.

\begin{proposition}\label{comp2mis}
From an expression of the integer $m$ with $n$ $1$s one can construct a graph on $n$ vertices with $m$ MISes.
\end{proposition}

\begin{figure}
\begin{center}
\psset{xunit=0.007in, yunit=0.007in}
\psset{linewidth=0.02in}
\begin{pspicture}(0,0)(120,80)
\pscircle*(0,20){0.04in}
\pscircle*(0,60){0.04in}
\pscircle*(40,20){0.04in}
\pscircle*(40,60){0.04in}
\pscircle*(80,40){0.04in}
\pscircle*(120,20){0.04in}
\pscircle*(120,60){0.04in}
\psline(0,20)(0,60)
\psline(40,20)(40,60)
\psline(40,20)(120,60)
\psline(40,60)(120,20)
\psline(120,20)(120,60)
%\rput[c](0,0){$x_1$}
%\rput[c](0,80){$x_2$}
%\rput[c](40,0){$x_3$}
%\rput[c](40,80){$x_4$}
%\rput[c](80,60){$x_7$}
%\rput[c](120,0){$x_6$}
%\rput[c](120,80){$x_5$}
\end{pspicture}
\end{center}
\caption{The construction described in the proof of Proposition~\ref{prop-integer-complexity}, applied to the expression
$$
10=(1+1)((1+1)(1+1)+1).
$$
There are graphs on $7$ vertices with more MISes than the graph shown because $10$ is not the greatest integer of complexity $7$ ($12$ is).}
\label{prop-integer-complexity-figure}
\end{figure}
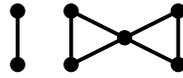

\begin{proof}
Before describing our inductive construction we need a definition.  Given graphs $G$ and $H$, their {\it join\/} is the graph $G+H$ obtained from their disjoint union $G\cup H$ by adding all edges connecting vertices of $G$ with vertices of $H$.  We know already from Proposition~\ref{mis-product} that $m(G\cup H)=m(G)m(H)$, and a similar formula for joins is easy to verify: $m(G+H)=m(G)+m(H)$ because every MIS in $G+H$ is either an MIS of $G$ or an MIS of $H$.

Now suppose we have an expression of the integer $m$ with $n$ $1$s.  If $n=1$, then there is only one such expression, $1$, and we associate to this expression the one vertex graph.  If $n\ge 2$, then any such expression must decompose as either $e_1+e_2$ or $e_1e_2$, where $e_1$ and $e_2$ are expressions with fewer $1$s.  If our expression is $e_1+e_2$ then we associate it to the join of the graphs associated to $e_1$ and $e_2$, and if our expression is $e_1e_2$ then we associate it to the disjoint union of the graphs associated to $e_1$ and $e_2$.  Figure~\ref{prop-integer-complexity-figure} shows an example.  It follows that the resulting graph has precisely as many vertices as the expression has $1$s, and precisely $m$ MISes.
\end{proof}

\minisec{Acknowledgements}
I am grateful to the referees for their detailed and insightful comments.  In particular, the change from ``separating families'' to ``separating covers,'' which simplified many of these results, was suggested by one of the referees.

\bibliographystyle{acm}
\bibliography{../refs}

\end{document}